%% file: CP2026III.tex
\documentclass{amsart}

\usepackage[utf8]{inputenc}%(only for the pdftex engine)
\usepackage{microtype}

\usepackage[backref,bookmarks=true,colorlinks=true,pdfstartview=FitV,linkcolor=blue, citecolor=red,urlcolor=green]{hyperref}  % hyperref package

\usepackage{mathrsfs}

\usepackage{amssymb,amsmath,amscd}
\usepackage{framed}

 \usepackage{stmaryrd}
\usepackage{url}

\usepackage[numbers]{natbib} 

\usepackage{subfig}
\usepackage{graphicx}
\usepackage{lpic}
\usepackage{etoolbox}
\usepackage{bbm} 

\makeatletter
\patchcmd{\lp@dimens}
{\epsfig{figure=\lp@epsfile}}
{\includegraphics{\lp@epsfile}}
{}{}
\patchcmd{\lp@dimens}
{\epsfig{figure=\lp@epsfile, width=\lp@tmpx, height=\lp@tmpy}}
{\includegraphics[width=\lp@tmpx,height=\lp@tmpy]{\lp@epsfile}}
{}{}
\makeatother

\usepackage[export]{adjustbox}

\usepackage{xcolor}

\numberwithin{equation}{section}

\input{lstylesCDG}

	\newtheorem{theorem}{Theorem}[section]

\newtheorem{lemma}[theorem]{Lemma}
\newtheorem{corollary}[theorem]{Corollary}

\theoremstyle{definition}
\newtheorem{definition}{Definition}[section]
\newtheorem{hypothesis}{Hypothesis}[section]

\theoremstyle{remark}

  \newcommand{\rome}{\mathrm e}
   \newcommand{\romel}{{\mathrm{e}, \mathrm{l}}}
    \newcommand{\romec}{\mathrm{c}}
     \newcommand{\romecr}{{\mathrm{c}, \mathrm{r}}}
      \newcommand{\romecrf}{{\mathrm{c}, \mathrm{r}, \mathrm{f}}}

\begin{document}

%	\theoremstyle{plain}
	%\newtheorem{theorem}{Theorem}[section]

%\begin{frontmatter} 

\title[The stability of Margulis space-times]{The stability of Margulis space-times with parabolic holonomy elements}

%\author[1]{Suhyoung}{Choi}{S.~Choi}{Daejeon}
%\author{Todd}{Drumm}{T.~Drumm}{Washington DC.}
%\author{William}{Goldman}{W.~Goldman}{College Park, Maryland}

\author{Suhyoung Choi} 
\address{Department of Mathematical Sciences, KAIST,
		305-701, Daejeon, Republic of Korea}
\email{schoi@math.kaist.ac.kr}
%\thanks{Choi was supported by the Mid-career Researcher Program through 
%the NRF grant NRF-2013R1A1A2056698 funded by the MEST.} %(No. 2009-0057445).} 

\thanks{This work was supported by 
the National Research Foundation of Korea(NRF) grant funded by the Korea government 
(MEST) (No. NRF-2022R1A2C3003162).}

\begin{abstract}
	Let $\Lspace$ be a flat Lorentzian space of signature $(2,1)$. A Margulis space-time is a noncompact complete flat Lorentzian $3$-manifold $\Lspace/\Gamma$, where the holonomy group $\Gamma$ is a free group of rank $\bg\ge 2$ acting freely and properly discontinuously by isometries.
	We consider the case where $\Gamma$ contains a parabolic element. We show that sufficiently small deformations of $\Gamma$ still act properly discontinuously on $\Lspace$ provided their linear parts are Fuchsian; moreover, the number of conjugacy classes 
of parabolic elements may increase or decrease under deformation.
	Our proof combines our previous compactification of $\Lspace/\Gamma$ relative to 
parabolic holonomy elements with a partial generalization of the work of Carri\`ere. However, this result depends only on the parts on parabolic actions of our earlier work. We believe that the shortness of the proof of this openness result is of independent interest. 
\end{abstract}
	
%	Mathematics Subject Classification 2010: 57M50 (primary), 83A99 (secondary); 
	
%}

%\.dedicatory{Dedicated to the memory of Todd Drumm.}

%\subjclass[2010]{Primary }

%\keywords{}

\keywords{geometric structures, flat Lorentz space-time, Margulis space-time, $3$-manifolds}

\subjclass[MSC 2020]{57M50,  83A99}
%  \journalname{Forum Mathematicum}

% \communicated{}
%\dedication{}
%\received{}
%\accepted{}
%\journalyear{2019}
%\journalvolume{}
%\journalissue{}
%\startpage{1}
%\aop
%\DOI{}
%\tableofcontents

\date{\today}

\maketitle 

%\begin{document}

\section{Introduction} 
Let $\Lspace$ denote $\bR^{2, 1}$ viewed as an affine space 
with the standard coordinate system, 
Lorentzian inner product, orientation, and time-orientation. 
Let $\Gamma$ be a discrete free subgroup of $\Isom(\Lspace)$ of rank $n\geq 2$ acting properly discontinuously and freely on $\Lspace$. 
Then $\Lspace/\Gamma$ is said to be a {\em Margulis space-time}, and 
$\Gamma$ is the holonomy group. (See Section \ref{sub:geostr}.)

Every element of $\Isom(\Lspace)$ is of the form $x \mapsto Ax + b$ for $A \in \SO(2, 1)$ and $b\in \bR^{2, 1}$. 
There is a homomorphism $\mathcal{L}:\Isom(\Lspace) \ra \SO(2, 1)$ taking the linear part of the affine transformations. 
Denote by $\SO(2, 1)_+$ the subgroup of elements preserving the future cone. 
We denote by $\Isom_+(\Lspace)$ the inverse image of $\SO(2, 1)_+$ in $\Isom(\Lspace)$.

By Section 2.12 of \cite{FG83}, $\mathcal{L}|\Gamma$ is an injective map
to a discrete group since a free group of rank $n\geq 2$ is never virtually solvable. 
Since $\mathcal{L}(\Gamma)$ is a Fuchsian group, it is well known that 
there can be only finitely many conjugacy classes of parabolic elements. 
(See Section \ref{sub:rcomp}.)

%We denote by $d_{\Isom(\Lspace)}$ the metric induced by the left-invariant Riemannian metric on $\Isom(\Lspace)$. This is a geodesic metric space in the sense that any two points has a shortest path between them. 

%We can consider $\SO(2, 1)$ acting on a hyperbolic space using the Klein model by projectivizing. 
%Note by the standard complete 
%hyperbolic surface theory due to Fricke that for any free subgroup 
%$\Gamma$ of rank $n$, $n\geq 1$, of $\Isom(\Lspace)$, 
%we can deform it to one without 
%parabolics by considering the quotient of the hyperbolic plane under 
%$\Gamma$ as a complete hyperbolic surface and deforming. 
%Actually, the subspace where there are parabolics is an analytic subspace of 
%codimension $i$ where $i$ is the number of parabolic cusps. 
%Since parabolics can only occur for elements parallel to ends, we need to 
%remove some finitely many subspaces of the Fricke space to obtain 
%Fuchsian groups without parabolics. 
%(See Lemma \ref{lem:deformnop}.)  

A Margulis space-time $M$ whose holonomy group contains primitive 
parabolic elements  $\eta_1, \dots, \eta_m$ can be compactified relative to these parabolic elements if we can add a totally geodesic real projective surface
and remove a union of disjoint solid tori with these parabolic elements as holonomies to obtain a compact $3$-manifold whose interior is homeomorphic to $M$.
(See Definition \ref{defn:rcomp}.) This is a fairly strong condition implying tameness
and existence of mutually disjoint end neighborhoods for parabolic holonomy elements. 

A {\em solid torus} in this article (and \cite{CDG22})
is a topological space homeomorphic 
to an annulus times an interval that can be a closed one or a 
half-open one. 

Let $F_n$ be a free group generated by $\gamma_1, \dots, \gamma_n$ 
for $n \geq 2$. 
Recall that we can identify $\Hom(F_n, \Isom(\Lspace))$ with 
$\Isom(\Lspace)^n$ by sending each representation $h$ to 
$(h(\gamma_1), \dots, h(\gamma_n))$.  We fix this identification in this article.
We denote by
$\mathcal{F}^{\rome}_n\subset \Isom(\Lspace)^n$ the subset 
where the projectivization of the linear part of the isometry group generated by 
$(g_1, \dots, g_n) \in \Isom(\Lspace)^n$
acts properly discontinuously 
and freely on the hyperbolic plane. 
We will denote by $\mathcal{F}^{\romel}_n$ the subset of $\mathcal{F}^{\rome}_n$ 
where no element of the group generated by 
$(g_1, \dots, g_n)$ is parabolic. We also regard these as
subsets of $\Hom(F_n, \Isom(\Lspace))$ as well. 

These are spaces with surjective maps to Fricke (or Teichm\"uller) spaces,  
which are well-known to be semialgebraic sets with open dense subsets 
homeomorphic to disjoint unions of cells. See
the beginning of Section \ref{sec:deforming}.
(See \cite{G09} also.) 

We present a short proof of the following openness: 

%We prove the following connectivity to Margulis groups with no parabolics.
\begin{theorem}[Deformation of Margulis space-times] \label{thm:main} 
Let $\Gamma$ be a discrete free group in $\Isom(\Lspace)$ of rank $n \geq 2$ containing a parabolic element. 
Suppose that $\Gamma$ acts freely and 
properly discontinuously on $\Lspace$ so that 
$\Lspace/\Gamma$ is a manifold, and admits a compactification 
relative to parabolic holonomy elements.
Fix free generators $g_1,\dots, g_n$ of $\Gamma$.
Then 
%for any path $g_i(t), i=1, \dots, n$ in $\Isom(\Lspace)$, so that $g_i(0)= g_i$ so that 
there exists a neighborhood $N \subset \mathcal{F}^{\rome}_n$ 
of $(g_1, \dots, g_n) \in \Isom(\Lspace)^n$, 
for each $\mu=(g'_1, \dots, g'_n)\in N$, 
the group $\Gamma_\mu$ 
generated by $\mu$
is a free group of rank $n$ acting properly discontinuously and freely on $\Lspace$. 
Moreover, $\Lspace/\Gamma_\mu$ for $\mu \in N$ 
admits a compactification relative to parabolic holonomy elements. 
%Then there exists a path  $g_i(t)$ for $t \in [0, \eps)$ for some $\eps> 0$  so that 
%$g_i(0)= g_i$ %depending on $\Gamma$ 
%so that $\Gamma_t$ generated by $g_i(t)$, $i=1, \dots, n$, still acts properly
%discontinuously and freely on $\Lspace$ and has no parabolics for $t> 0$. 
%Also, $\Gamma_t$ is a free group  with generators $g_i(t), i=1, \dots, n$ for $t\in [0, \eps)$. 
%so that for  $\Gamma'$ in  $\Isom(\Lspace)$ without parabolic isomorphic to 
%$\Gamma$ with generators $g'_1,\dots, g'_n$ where 
%$d_{\Isom(\Lspace)}(g_i, g'_i) < \eps$ for $i=1, \dots, n$, 
%$\Gamma'$ still acts freely and properly on $\Lspace$.
\end{theorem}

Our major tool is the following generalization on the Marcus conjecture:
\begin{theorem}[Generalization of the completeness of Carri\`ere \cite{Carr89}]\label{thm:Carriere} 
Let $M$ be a flat Lorentzian manifold with free holonomy group of rank $\geq 2$. 
Suppose that $M$ admits a compactification relative to parabolic holonomy elements. 
Then $M$ is complete. 
\end{theorem}

Note that we can surprisingly increase or decrease the number of cusps in the groups.
The final statement says that the subspace in $\mathcal{F}^{\rome}_n$ 
of relatively compactified Margulis space-times is open.
%Also, the result supports the crooked decomposition conjecture of 
%Charette, Drumm, and Goldman since Proposition \ref{prop:cr} holds. 

In \cite{CDG22}, we proved that the relative compactification always exists. 
However, the present paper depends only on the results about parabolic regions in Section 3 and Appendix A of \cite{CDG22}, and not on the full set of results. 
%Also, if a finite number of mutually disjoint crooked planes decompose 
%a Margulis space-time, then a relative compactification will exist.
%(See Proposition \ref{prop:crirc}.)

When $\Gamma$ has no parabolic elements, this openness follows from 
the main theorem of Goldman-Labourie-Margulis \cite{GLM09}.
We prove the corresponding statements 
in Theorem \ref{thm:nonpara} below. 
%However, we question the openness of the property of proper actions for 
%$\mathcal{F}^{\rome}_n$ since 
Our result is moderately surprising since Properness 
Criterion 1.1 given in \cite{CDG22} does not
correspond to conditions based on compact sets. 
Our result suggests that 
we might be able to find a compact set of criteria. 
%See also Remark \ref{rem:paradef}.

%Choi-Goldman \cite{CG17} and 
%Danciger, Gu\'eritaud, and Kassel \cite{DGK}, \cite{DGK16}.  

We note the announcement by Danciger, Kassel, and Gu\'eritaud \cite{DGKp} 
proving the Charette-Drumm-Goldman conjecture that 
$\Gamma$ as above has a fundamental domain bounded by crooked planes.
If $\Gamma$ admits no parabolic element, then the conjecture 
has already been proved by the same people \cite{DGK16}. 
%We believe that this would imply Theorem \ref{thm:main}
%by Proposition \ref{prop:cr}. 
Our approach is different in that we use the relative compactification, 
and we do not address this conjecture here. Also, we 
take the point of view of 
the geometric structures where we do not assume that the action is proper.  
%Define Crooked plane here.. 

If the conjecture holds, our result also follows by perturbing the fundamental
domain bounded by crooked planes and Theorem 3.5 of Drumm \cite{Drumm92}.
Hence, our result supports the above conjecture.
Also, a crooked plane decomposition implies the existence of relative compactification. 
In another paper, we show that a Margulis space-time admits a decomposition by 
quasi-disjoint crooked planes.  These results are currently 
in preparation.

%Let $\Lspace$ be the affine space 
%associated with the oriented vector space  
%$\bR^{2, 1}$ with the standard $(2, 1)$-signature bilinear form

We give an outline of the proof. 
In Section \ref{sec:preliminary}, we review 
the parabolic actions on $\Lspace$, affine, flat Lorentzian, 
and real projective structures on manifolds, 
and relative compactification of Margulis-space times. 

In Section \ref{sec:deforming}, we discuss 
the deformations of Margulis space-times. 
%For deformations of Margulis space-times with parabolics, we use the earlier result from %\cite{CDG22} and 
We use the relative compactification $M_{\romec}$ 
of our Margulis space-time $M$ to obtain a complete manifold. 
This means adding a real projective surface $\Sigma$ 
which is a union of two isometric 
complete hyperbolic surfaces $\Sigma_+$ and $\Sigma_-$ with cusps and some annuli. 
We remove the solid tori corresponding to parabolic holonomy elements 
to obtain a compact $3$-manifold $M_{\romecr}$.
The boundary of each solid torus is an annulus ending in the boundary of 
a cusp neighborhood in $\Sigma_+$ and the corresponding 
cusp neighborhood in $\Sigma_-$. 
(See Definition \ref{defn:rcomp}.)

%See Section \ref{sub:before} for this. 

Then we deform $\Gamma$:
We deform the solid tori  
first in Section \ref{sub:deformtori}. 
We deform $M_{\romecr}$ 
to a compact $3$-manifold $M'_{\romecr}$ with corresponding holonomy
using the theory of deformation of geometric structures by Ehresmann-Thurston principle. 
We glue back the solid tori accordingly deformed to obtain
a compact $3$-manifold $M'_{\romecrf}$ with a real projective structure
in Section \ref{sub:deform}. 
We complete the proof of Theorem \ref{thm:main} by 
Theorem \ref{thm:Carriere} proved in the next subsections.
%The boundary can be filled in.
%The boundary of the solid tori now becomes convex surface ending at two closed geodesics. 
%The amount that you can deform is determined by the boundary still bounding 
%the solid tori. 
%\marginpar{\tiny This is just standard} 
%Then $M_{\romecr}$ to $M'_{\romecr}$ using the boundary deformation here which 
%uses the work of Lok. 

%to obtain a closed real projective surface. 
%We obtain 

Now, we prove Theorem \ref{thm:Carriere} in Sections \ref{sub:Carriere}
and \ref{sub:generalcase}.
The universal cover of the relative compactification $M_{\romec}$ 
inherits a real projective structure. 
We show that the interior $M$ must be convex 
and complete. 
The argument is not substantially different from that of Carri\`ere: 
we still have a compact domain that 
we can use in place of a fundamental domain. 
Then a consideration of the limit set shows that the interior is 
a Margulis space-time. 

%This proves Theorem \ref{thm:main}. 

\subsection{Acknowledgment} 
We thank Virginie Charette, Jeffrey Danciger, Todd Drumm, and William Goldman
for many discussions.

\section{Preliminary} \label{sec:preliminary}

\subsection{Actions of parabolic groups} \label{sub:action}

A {\em parabolic element} is an element of $\Isom(\Lspace)$ whose linear
part has a single eigenvalue and is not diagonalizable. 

\begin{definition} 
	Let $N$ be a nilpotent skew adjoint endomorphism. 
	We will call a frame $\va, \vb, \vc \in \bR^{2, 1}$ satisfying the following properties: 
	\begin{itemize} 
		\item $\vb = N(\va), \vc = N(\vb)$.
		\item $\va, \vc$ are null, and $\vb$ is a unit-norm space-like vector.  
		\item $\Bs(\va, \vb)=0=\Bs(\vb, \vc), \Bs(\va, \vc) = -1$. 
	\end{itemize}
	an {\em adopted frame} of $N$.  %\hypertarget{term-pcs}{{\em adopted frame}} of $N$.
	We say that $N$ is {\em accordant} if an adopted frame has the standard orientation. 
\end{definition}
Note that the orientation of an adopted frame is determined by $N$. Moreover, $-N$ is accordant if and only if $N$ is not accordant.

\begin{corollary}[Corollary 3.8 in \cite{CDG22}] \label{cor:parac}
	Let $N$ be a nilpotent skew adjoint endomorphism. Then
the Lorentzian vectors $\va, \vb, \vc$ 
 satisfying the properties 
\begin{itemize} 
\item  $\Bs(\va, \vb)=0=\Bs(\vb, \vc), \Bs(\va, \vc) = -1$, 
\item $ \vc = N(\vb), \vb = N(\va)$, and 
\item $\vb$ is a unit space-like vector, $\vc \in \ker N$ is causally null,  and $\va$ is null
\end{itemize}
are determined up to changes 
$\vb \ra \vb + c_0 \vc, \va \ra \va + c_0 \vb + \frac{c_0^2}{2} \vc$, $c_0\in \bR$.
%with respect to 
%%the skew-symmetric nilpotent endomorphism $N$ and 
$\Bs: V \times V \ra \bR$.  
%\marginpar{\tiny Nov 4, second "with respect to" changed to "and"}
Furthermore, the adopted frame for $N$ is determined only up to these changes. 
\end{corollary} 

In an adopted frame, 
the bilinear form $\Bs$ has matrix
\begin{equation}\label{eqn:Bform}  
	\left(
	\begin{array}{ccc}
		0 & 0 & -1 \\
		0 & 1 & 0 \\
		-1 & 0 & 0
	\end{array}
	\right).
\end{equation}

Let $\gamma$ be a parabolic transformation $\Lspace \ra \Lspace$.
Then it is in a one-parameter subgroup 
	\begin{equation}\label{eqn:Phit0}
		\Phi(t) := \exp t
		\left(
		\begin{array}{cc}
			N & \vec{v}_\gamma \\
			0 & 0 
		\end{array}
		\right), t\in \bR, \vec{v}_\gamma \in \bR^{2, 1},
\end{equation} 
for an accordant nilpotent skew adjoint endomorphism $N$. 
In the adopted frame with a choice of origin, we
can write 
\begin{equation}\label{eqn:Phit}
	\Phi(t) := \exp t
	\left(
	\begin{array}{cccc}
		0  & 1  & 0  &0\\
		0  & 0  & 1  & 0\\
		0  & 0  &  0 & \mu \\
		0 &  0 & 0 & 0 
	\end{array}
	\right)
	= 
	\left(
	\begin{array}{cccc}
		1 & t  & t^{2}/2 & \mu t^{3}/6  \\
		0   & 1  & t  & \mu t^{2}/2  \\
		0  &  0  & 1 & \mu t \\ 
		0 & 0 & 0 & 1   
	\end{array}
	\right)
\end{equation} 
for $\mu \in \bR$. 
This was proved in Section 3.1 of \cite{CDG22}. 

\subsection{Geometric structures on manifolds} \label{sub:geostr}

A pair $(G, X)$ of a Lie group $G$ acting on the space $X$ faithfully and 
transitively is called the $(G, X)$-geometry. 
A $(G, X)$-structure on a manifold $M$ is a maximal atlas of charts to $X$ 
such that 
the transition maps are in $G$. 

The geometry $(\Aff(\bR^n), \bR^n)$ for $n \geq 1$ is 
called the {\em affine geometry}, and a geometric structure on a manifold modeled on
the pair is called an {\em affine structure}. 

The geometry $(\Isom(\Lspace), \Lspace)$ is 
called the {\em Lorentz geometry}, and a geometric structure on a manifold modeled on
the pair is called a {\em flat Lorentzian structure}. Clearly, a flat Lorentzian structure 
give us a canonical affine structure of dimension $3$.  

The geometry $(\SL_\pm(n+1, \bR), \SI^n)$ for $n\geq 1$ 
is called the {\em real projective geometry},
and a geometric structure on a manifold modeled 
on the pair is called a {\em real projective structure}. 

We have inclusions
\[        
(\Isom(\Lspace), \Lspace)\subset (\Aff(\bR^3), \bR^3)
\subset (\SL_\pm(4, \bR), \SI^3).
\]
Hence, a flat Lorentzian manifold, and so a Margulis space-time, 
can be considered an affine $3$-manifold 
and also a real projective $3$-manifold. 

Let $N$ be a $n$-dimensional 
manifold with a universal cover $\tilde N$ and a deck transformation 
group $\pi_1(N)$. 
Given an affine (resp. projective) structure on $N$, there exists an immersion 
$\dev: \tilde N \ra X$ for $X = \bR^n$ (resp. $X= \SI^n$) and a homomorphism 
$h: \pi_1(N) \ra \Aff(\bR^n)$ (resp. $\ra \SL_\pm(n+1, \bR)$) satisfying 
$\dev \circ \gamma = h(\gamma) \circ \dev$ for each $\gamma \in \pi_1(N)$. 
Here, $\dev$ is called a {\em developing map}, 
$h$ is called a {\em holonomy homomorphism}, the image of $h$ is 
called a {\em holonomy group}, and $h(\gamma)$ for a closed curve 
$\gamma$ is called the {\em holonomy} of $\gamma$. 
(For details, see \cite{psconv}.) 

Given two affine manifolds $N_1$ and $N_2$, 
an affine map $N_1 \ra N_2$ is a map that locally sends a chart to a chart. 
A {\em complete affine manifold} is an affine manifold 
affinely diffeomorphic to a quotient manifold of $\bR^n$ under a properly discontinous and free action 
of an affine group. An affine manifold is complete if and only if 
it is affinely diffeomorphic to $\bR^n/\Gamma$ for a holonomy group $\Gamma$.

A {\em directed ray} in an affine (resp. projective) manifold is an arc
equipped with an orientation which maps to an affine (resp. projective) geodesic under the composition with the developing map.

\subsection{Definition of compactification relative to parabolic holonomy elements}\label{sub:rcomp}
The space of directed rays in $\bR^4$ is identified with $\SI^3$.
Let $\mathcal{A}: \SI^3 \ra \SI^3$ be a map sending $x$ to $-x$. 
For any set or point $X$, we denote the image $\mathcal{A}(X)$ by $X_-$.
Let us give coordinates $x_i, i=1, \dots, 4$ in $\bR^4$. 
The Lorentzian space $\Lspace$ is identified with the upper open hemisphere
given by $x_4> 0$, or the hyperplane given by $x_4=1$.
$\Lspace$ has coordinate functions $x_1, x_2, x_3$
with the Lorentz form $x_1^2 + x_2^2 - x_3^2$. 
This identification map is the stereographic projection from the origin.
%Also, twice the Euclidean metric based on the coordinates $x_1, x_2, x_3$ 
%dominates the spherical metric in the identification as we can verify by 
%a rough computation. 

The boundary of $\Lspace$ in $\SI^3$ 
is a $2$-sphere to be denoted by $\Ss$.
The space of future time-like directions is 
an open disk to be denoted by $\Ss_+$ 
and the space of past directions is a disk to be denoted by $\Ss_-$. 
We consider $\Isom(\Lspace)$ as a subgroup of 
$\SL_\pm(4, \bR)$ acting on $\SI^3$.

We regard $\Ss_+$ as the Klein model of the hyperbolic plane. 
Now $\SO(2, 1)$ acts on $\Ss_+\cup \Ss_-$, and $\SO(2, 1)_+$ acts 
on $\Ss_+$ as the isometry group of $\Ss_+$ with the hyperbolic metric. 
Also, the projectivization $\SO(2, 1)/\pm \Idd$ of $\SO(2, 1)$ acts on $\Ss_+$ 
as the full hyperbolic isometry group. 
We can write $\Ss = \Ss_+ \cup \Ss_- \cup \Ss_0$ as a disjoint union where 
$\Ss_0$ is the space of directions of the directed space-like rays and 
directed null rays.
$\Ss$ inherits a boundary orientation from
the closure of $\Lspace$, this induces an orientation on 
$\Ss_+$, and the circle 
$\partial \Ss_+$ is given the boundary orientation from $\Ss_+$, 
where $\partial \Ss_+$ is the topological boundary of $\Ss_+$ in $\Ss$. 
%$\Ss_+$ and $\Ss_-$, and $\Ss_0$.  

{\em Geodesic segments} in $\Ss$ are connected arcs in great circles. 
For $x \in \partial \Ss_+$, we define $\zeta(x)$ to be the closed 
geodesic segment in $\Ss_0$ from $x$ to $x_-$ tangent to 
$\partial \Ss_+$ in the oriented direction.  
For any arc $\alpha$ in $\partial \Ss_+$, we define $Z(\alpha)$ as 
the union 
of $\zeta(x)$ for $x\in \alpha$. This is a disk with two boundary components 
$\alpha$ and $\alpha_-$. 

Let $\Gamma$ be the image of $\rho:F_n \ra \Isom(\Lspace)$ acting properly discontinuously
on $\Ss_+\cup \Ss_-$. 
Since $\SO(2,1)$ has two components, there is 
a subgroup $\Gamma'$  of index at most $2$, which 
acts on each of $\Ss_+$ and $\Ss_-$. 
Let $\Sigma_+ := \Ss_+/\Gamma'$ and $\Sigma_- := \Ss_-/\Gamma'$,
which are real projective surfaces. 
$\Gamma'$ acts as a Fuchsian group acting on the hyperbolic plane 
$\Ss_+$ and 
acts properly discontinuously and cocompactly 
on $\Ss_+ \cup \beta$ for a union $\beta$ of arcs in $\partial \Ss_+$. 
Then $\Sigma_+$ is an open dense surface in a surface 
$\hat \Sigma_+:= (\Ss_+\cup \beta)/\Gamma'$ with cusp ends corresponding 
to the parabolic elements in $\Gamma'$. 
The union $\beta$ covers a finite number of boundary components $\beta_1, \dots, \beta_p$ of $\hat \Sigma_+$. 
We choose a single lift $\tilde \beta_i$ for each $i$. 
$\beta$ is a union of images of $\tilde \beta_i, i=1, \dots, p$ under $\Gamma'$. 
%There are also mutually disjoint cusp neighborhoods $E_i$, $i=1, \dots, k$. 
%For each $i$, we choose a horoball $\tilde E_i$ covering $E_i$, $i=1, \dots, k$. 
%Also, $E_i$, $i=1,\dots, k$, form cusp neighborhoods in $\Sigma_+$, and there are corresponding 
%cusp neighborhoods $E_{i, -}$ in $\Sigma_-$ for $i=1, \dots, k$. 
%We denote by $E$ the union of $E_i, i=1, \dots, k$ and $E_-$ the union of $E_{i,-}$ for $i=1, \dots, k$. 

Let $\partial \Ss_+$ denote the boundary of $\Ss_+$ 
equipped with its boundary orientation.
Then there is a domain $\tilde D_i= Z(\tilde \beta_i)$ in $\Ss_0$ for $i=1, \dots, p$ that is a union of 
maximal geodesics in $\Ss_0$ tangent to $\partial \Ss_+$ in the boundary direction. 
We also have 
$\partial \tilde D_i = \tilde \beta \cup \tilde \beta_{i, -}$. 
Then $\Gamma$ acts properly discontinuously on 
$\bigcup_{i=1,\dots, p}\Gamma(\tilde D_i)$. 
On each $\tilde \beta_i$, there is a unique primitive element $\kappa_i$
acting on it in the orientation direction, 
which also acts on $\tilde D_i$. 

Define $\tilde \Sigma := \Ss_+ \cup \Ss_- \cup \bigcup_{i=1}^p \Gamma({\tilde D_i})$.
Then $\Gamma$ acts properly discontinuously, 
and we set $\Sigma := \tilde \Sigma/\Gamma$.  
We say that $\Sigma$ is the {\em canonical ideal real projective surface} for $\Gamma$. 
(See Section 5.1 of \cite{CDG22} and Theorem 5.3 of \cite{CG17}.)
Then $\Sigma$ contains the image of union of 
$\Sigma_+$ and $\Sigma_-$, and its complement is 
a union of annuli or M\"obius bands $A_i$ for $i=1, \dots, p$, each of 
which is an image of ${\tilde D_i}$.

We define $\Sigma' := (\Ss_+ \cup \Ss_-)/\Gamma$. 
Then $\Sigma$ is a union of 
\[\Sigma' \hbox{ and } A_i, i=1, \dots, p,\] 
where $\Sigma'$ is isometric to the union of one or two copies of $\Sigma_+$. 
%\begin{itemize} 
%\item $\Sigma_+$,
%\item $\Sigma_-$, 
%\item $A_i$, $i=1, \dots, m$. 
%\end{itemize}

A {\em bordification} of an open $n$-manifold $N$ by an $(n-1)$-manifold $N'$
is a manifold $N''$ whose manifold boundary equals $N'$ and its interior equals $N$. 

\begin{definition}[Compactification relative to parabolic holonomy elements] \label{defn:rcomp}
We say that a flat Lorentzian manifold $M$ with parabolic holonomy elements 
admits a {\em compactification relative to parabolic holonomy elements} if 
the following hold:
\begin{itemize}
%\item $\pi_1(M)$ projectively acts on the universal cover $\tilde M$ of $M$ as 
%a deck transformation group. 
\item $\pi_1(M)$ acts properly discontinuously and freely on
the bordification $\tilde M \cup \tilde \Sigma$ of $\tilde M$ 
(resp. $\tilde M \cup \mathcal{A}(\tilde \Sigma)$), and the quotient is a union of $M$ and a surface $\Sigma$. 
\item The real projective structure on $\tilde M \cup \tilde \Sigma$ extends 
that of $\tilde M$ and $\pi_1(M)$ acts projectively on $\tilde M \cup \tilde \Sigma$. 
\item We have a pairwise disjoint collection of solid tori
$T_1, \dots, T_m$ closed in $M_\romec:= M \cup \Sigma$, 
whose holonomy groups are parabolic and contain
representatives from each of the conjugacy classes of all parabolic elements. 
%for two complete hyperbolic surfaces $\Sigma_+, \Sigma_-$ with cusps  and 
\item A compact $3$-manifold $M_\romecr:= M_\romec - \bigcup_{i=1}^m T_i^o$ is a deformation retract of $M_\romec$. 
\end{itemize}
\end{definition}
%Note that $M_c$ has a real projective structure since each element of $\Gamma$ %is 
%considered to be in $\SL_\pm(4, \bR)$. 

We may also assume that the topological 
boundary of $T_i$ in $M_\romec$ for each $i$  is a surface ruled by time-like geodesics and invariant 
under a one-parameter group of parabolic elements. 
This can be arranged, if necessary, by Theorem 3.14 of \cite{CDG22}:
This procedure can be accomplished inside any such torus with a parabolic holonomy group
since each component of the preimage of $T_i$ in $\Lspace \cup \tilde \Sigma$ is a neighborhood of $\zeta(p)$ for a fixed point $p$ of a parabolic element of $\Gamma$. 
(See Section 3.3, 
Appendix A of \cite{CDG22}, and Section 5.1.1 of \cite{CDG22}.)

 Suppose that
 $\Gamma \subset \Isom(\Lspace)$ acts properly discontinuously
    and freely. Then Margulis invariants and Charette-Drumm invariants 
    are either all be positive or all negative by Theorem 4.1 of Charette-Drumm \cite{CD05}.

    If the invariants are positive, then we use $\tilde \Sigma$ 
    in Definition \ref{defn:rcomp}; otherwise, we use 
    $\mathcal{A}(\tilde \Sigma)$.

    If all the invariants are negative, we reverse the orientation of 
$\Lspace$ so that the invariants become positive. 
    We assume the following for the rest of the paper. 
\begin{hypothesis}[Positive Invariants]\label{hyp:positive}
   If a subgroup  $\Gamma \subset \Isom(\Lspace)$ acts properly discontinuously and freely on
   $\Lspace$, then all Margulis invariants and Charette-Drumm invariants are positive. 
\end{hypothesis}

\section{Deforming} \label{sec:deforming} 

We will prove Theorem \ref{thm:main} in this section. 
Let a Margulis group $\Gamma \subset \Isom(\Lspace)$ be 
freely generated by $g_1, \dots, g_n$ satisfying Hypothesis \ref{hyp:positive}.
%Since $\Gamma$ is free, we may suppose that 
%$g_1, \dots, g_n$ are free generators. 
%We suppose that $\Gamma$ has parabolic elements. 
%We will identify $\pi_1(M)$ with $\Gamma$, and 
%hence with the free group $F_n$ of rank $n$. 
%Note that $\Hom(F_n, \Isom(\Lspace))$ can be identified with 
%$\Isom(\Lspace)^n$ by sending 
%\[\rho \mapsto (\rho(g_1), \rho(g_2), \dots, \rho(g_n)).\] 
Since $\Hom(F_n, \Isom(\Lspace)) = \Isom(\Lspace)^n$ is a real algebraic set, 
we use the product topology, 
under which it is locally path-connected geodesic-metric 
space induced from a Riemannian metric.  For $\mu$ in $\Isom(\Lspace)^n$,  
we define $\Gamma_\mu$ as an image of the holonomy 
homomorphism $h_\mu: F_n \ra \Isom(\Lspace)$
where $(\gamma_1, \dots, \gamma_n) \mapsto \mu$. 

Let $\Pi: \SO(2, 1) \ra \PO(2, 1)=\SO(2, 1)/\{\pm \Idd\}$ denote the quotient homomorphism.
A representation $h: F_n \ra \Isom(\Lspace)$ is an {\em affine deformation of a Fuchsian representation} if
$\Pi \circ \mathcal{L}\circ h$ is a Fuchsian representation.

%this will prove our theorem. 

We note that the space of discrete faithful characters
\[\mathfrak{F}_n\subset \Hom(F_n, \PO(2, 1))/\PO(2, 1)\] 
is homeomorphic to a union of open cells with some boundary points.
We call $\mathfrak{F}_n$ the {\em Fricke subspace}. 
This is a semialgebraic set where a dense open set is a disjoint union of 
open cells, and 
boundary strata consist of representations with parabolic elements. 
Here only the boundary-parallel homotopy classes can map to parabolic elements by 
the theory of Fricke. (See \cite{G09} where this is essentially proved.) 
Of course, different open components 
consist of characters of representations of the fundamental groups of 
surfaces with distinct topological structures. 

We recall the standard cone neighborhood theory following from 
the triangulations of the semialgebraic sets as in Theorem 9.2.1 of  \cite{Bochnak}: 
We define a {\em cone over a semialgebraic set $X$} to be a subspace
homeomorphic  to a product $X\times [0, 1]$ with $X\times \{0\}$ 
collapsed to a point. The image of $X \times \{0\}$ is called the 
{\em cone point}.
A point $x$ of a semi-algebraic set has a neighborhood with a semi-algebraic homeomorphism to a cone $C$ over a semi-algebraic set $S$ 
with the cone point $x$. 
%We denote by $D'_n$ the subset $\mathcal{L}^{-1}(D_n)\subset 
%\Hom(F_n, \Iso(\Lspace))$. 

\begin{lemma} \label{lem:deformnop} 
Let $\rho:F_n \ra \Isom(\Lspace)$ be an affine deformation of a Fuchsian representation whose image contains parabolic elements.
%where the image of 
%$\mathcal{L}\circ \rho$ is a free group with a parabolic element. 
Then $\mathcal{F}^{\rome}$ and $\mathcal{F}^{\romel}$ are semialgebraic sets. 
There is a cone neighborhood $N$ of $\rho$ in $\mathcal{F}^{\rome}_n$ where 
the subset $N'$ of $N$ consisting of representations without parabolic image elements is 
an open dense cone, which is a complement of a semialgebraic subset of codimension 
$\geq 1$ in $N$ passing through $\rho$. 
%Then the subset of elements $x$ of $D'_n$
%$\Gamma_x$ has no parabolic is an open dense subset of $D'_n$, which is the complement of finite union of codimension $1$ semi-algebraic set. 
%there is a path of affine deformations of Fuchsian representations $\rho_t$, $t \in [0, 1]$ where,
%$\rho_0 = \rho$ and $\rho_t$ has no parabolic for $t> 0$. 
\end{lemma}
\begin{proof} 
%Let $\mathcal{F}^{\rome}_n \subset \Hom(F_n, \SO(2,1))$ is the inverse image of $\mathcal{F}_n$. 
The subset $\hat D_n$ of $\mathcal{F}^{\rome}_n$ of representations 
with some parabolic image elements is a union of finitely many semi-algebraic subsets of codimension $\geq 1$ passing through $\rho$
since only peripheral elements can have a parabolic holonomy. 
The subset in question is the inverse image of $\hat D_n$. 

%There is a map 
%$\mathcal{L}': \Hom(F_n, \Isom(\Lspace)) \ra \Hom(F_n, \PO(2, 1))$ by %%sending 
%$h(g)$ to $\Pi\circ \mathcal{L}(h(g))$ for $g \in F_n$. 
Define $\mathcal{L'} := \Pi\circ \mathcal{L}\circ h$.
This can be considered a fiber bundle
$(\Pi\circ\mathcal{L})^n:\Isom(\Lspace)^n \ra \PO(2, 1)^n$
with fibers that can be identified with $(\bZ_2\ltimes \bR^{2, 1})^n$. 
The proposition follows from the standard 
cone-neighborhood theory of semi-algebraic sets. 
%Let $\mathcal{F}$ denote the subspace of Fuchsian representations
%in $ \Hom(F_n, \PO(2, 1))$. This fibers over a Fricke space of hyperbolic surfaces 
%allowing cusps, which is a cell with nonempty boundary.
%
%Since it is a fiber bundle, 
%$\mathcal{F}$ is a manifold with boundary. 
%We have $h \in ({\mathcal{L}^n})^{-1}(\mathcal{F})$, and 
%$({\mathcal{L}^n})^{-1}(\mathcal{F}) \ra \mathcal{F}$ is again a fiber bundle with fibers diffeomorphic to  $(\bR^{2, 1})^n$. 
%%Hence, there is always a path of Fuchsian representation $\rho_t$ where 
%$\rho_t$ has no parabolics. 
\end{proof} 

%Suppose that $\mathcal{L}$ acts on a positive cone.  \marginpar{\tiny double covering: where to start?} 
%We may suppose that they are not parabolic. 

%We will use the Riemannian metric on $\Hom(F_n, \Isom(\Lspace))$ considered 
%as the product Riemannian metric on $\Isom(\Lspace)^n$, which is a geometric metric space. 

\subsection{Diffused Margulis invariants} 

Let us denote by $\mathfrak{F}_{n, C}$ the subspace of $\mathfrak{F}_n$ consisting 
of representations corresponding to convex cocompact Fuchsian groups. 
This is the complement of Fuchsian representations with parabolic elements in  $\mathfrak{F}_n$. 
%Let $\mathcal{F}^{\rome}_{n, C}$ be the subspace of 
%$\Hom(F_n, \Isom(\Lspace))$ 
%the inverse image of $\mathcal{F}_{n, C}$ under $\Pi \circ \mathcal{L}$. 
It is well-known that the subspace of $\Hom(F_n, \Isom(\Lspace))$ of representations 
corresponding to  $\mathfrak{F}_{n, C}$ is the open set 
$\mathcal{F}^{\romel}$.

%This theorem is well-known: 
\begin{theorem} \label{thm:nonpara}
The subset 
of $\mathcal{F}^{\romel}$ corresponding to proper actions is open 
in $\Hom(F_n, \Isom(\Lspace))$. 
\end{theorem} 
\begin{proof} 
%We first suppose that the linear parts are in $\SO_+(2,1)$. 
%We first suppose that $\mathcal{L}(\Gamma) \subset \Isom_+(\Lspace)$. 
%A representation $\Hom(F_n, \Isom(\Lspace))$ is 
%in one-to-one correspondence with the space of pairs of form $(\rho, \vec{v})$ where 
%$\rho$ is a representation to $\SO(2, 1)$ and $\vec{v} \in H^1(F_n, \bR^{2, 1}_\rho)$. 
The projection $\mathcal{L}$ induces 
a map $\Hom(F_n, \Isom_+(\Lspace)) \ra \Hom(F_n, \SO_+(2, 1))$, 
each of whose fibers can be identified to 
the space $Z^1(F_n, \bR^{2, 1}_\rho)$ of cocycles where $\rho$ 
is an element of the image. 
$\Hom(F_n, \Isom_+(\Lspace))$ is 
in one-to-one correspondence with the space of pairs of the form 
$(\rho, \vec{v})$ where 
$\rho$ is a representation to $\SO_+(2, 1)$ and 
$\vec{v} \in Z^1(F_n, \bR^{2, 1}_\rho)$. 
(See Section 2 of \cite{GLM09}.) 

Let us denote by $\mathcal{C}(\Sigma)$ the space of currents on a complete hyperbolic surface $\Sigma$ with a compact convex core.

For $h \in \mathcal{F}^{\romel}_n$, 
the diffused Margulis invariant function is a continuous 
$\bR$-valued function on
$\mathcal{C}(\Sigma)\times Z^1(F_n, \bR^{2,1}_{\rho})$
for the linear part $\rho$ of $h$.  
The subspace $\mathcal{P}(\Sigma)$ of $\mathcal{C}(\Sigma)$ consisting 
of currents with total measure $1$ is compact in the weak $\star$-topology. 
(See Section 6 of \cite{GLM09}.) 
The main theorem of \cite{GLM09} is that 
$(\rho, \vec{v})$ corresponds to an affine group acting properly on $\Lspace$ 
if and only if the diffused Margulis invariants for $(\phi, \vec{v})$ are positive for all 
$\phi \in \mathcal{P}(\Sigma)$. 

Notice that $Z^1(F_n, \bR^{2, 1}_\rho)$ is a vector space of dimension 
$3n$ independently of $\rho$. 
Since $Z^1(F_n, \bR^{2, 1}_\rho)$ depends continuously on 
$\rho$ when $\Pi\circ \rho \in \mathcal{F}_n$, the openness of the properness locus follows. 

For the general case, 
it is easy to show that $\Hom(F_n, \Isom(\Lspace))$ has two components, 
say $C_0$ and $C_1$, where $C_0$ is 
equal to $\Hom(F_n, \Isom_+(\Lspace))$.
Choose a representation $\rho: F_n \ra \SO(2, 1)$ in $C_1$. 
There is an index-two free subgroup $F'$ isomorphic to $F_{2n-1}$ of $F_n$ such that 
$\rho| F'$ maps into $\SO_+(2, 1)$. Here, 
$F'$ is just the subgroup of elements with even word lengths. 
There is an exact sequence $1 \ra F' \ra F_n \ra \bZ_2 \ra 1$. 
Also, $F_n = \langle F', \phi\rangle$ for some $\phi \in F_n$. 
We have a continuous restriction map 
$r:\Hom(F_n, \Isom(\Lspace)) \ra \Hom(F', \Isom_+(\Lspace))$. 
%since the subset where this restriction is defined is closed, and, moreover, 
%it is open by the continuity argument. 

Finally,
since we need to add only one image of $\phi$, 
it is straightforward to show that 
the subspace of proper actions in $\Hom(F_n, \Isom(\Lspace))$ is 
the inverse image of the subspace of proper actions in 
$\Hom(F', \Isom_+(\Lspace))$ under $r$. 
\end{proof}

\subsection{Before deformation} \label{sub:before} 
To begin, we suppose that $\mathcal{L}(\Gamma) \subset \Isom_+(\Lspace)$. 
%Let $\Gamma_t$ be the parameters of free groups given by 
%$g_1(t), \dots, g_n(t)$ where $g_i(0) = g_i$ for $i=1, \dots, n$, 
%which could be chosen to be a geodesic, and $t$ is the Riemannian distance from $g_i$. 
%Here, we assume that $g_i(t) \in \Isom_+(\Lspace)$. 
%We may suppose that 
%$d_{\Isom(\Lspace)}(g_i, g_i(t)) < \eps$ for $i=1, \dots, n$. 
%We will show that for sufficiently small $\eps$, 
%$\Gamma_t$ acts properly on $\Lspace$. 
%To prove Theorem \ref{thm:main}, it is sufficient to obtain an upper bound on $t$ since we
%can restrict to a Riemannian normal neighborhood of $(g_1, \dots, g_n)$.
%
Let $h$ be a proper affine deformation of a Fuchsian representation 
that has parabolic image elements.
We will use the notation of Section \ref{sub:rcomp} from now on. 
%Let $\Gamma$ denote its image. 
%We recall some facts from \cite{CDG22}. 
%Let $\Sigma_+ := \Ss_+/\Gamma$ and $\Sigma_- := \Ss_-/\Gamma$,
%which are real projective surfaces. 
%$\Gamma$ acts as a Fuchsian group acting on the hyperbolic plane $\Ss_+$ and 
%acts properly discontinuously and cocompactly 
%on $\Ss_+ \cup a$ for a union $a$ of arcs in $\partial \Ss_+$. 
%Then $\Sigma_+$ is dense in a complete hyperbolic surface 
%$\hat \Sigma_+:= \Ss_+\cup a/\Gamma$ with cusp ends corresponding 
%to the parabolic elements in $\Gamma$. 
%The union $a$ covers a finitely many boundary components $a_1, \dots, a_m$ of $\hat %\Sigma_+$. 
%We choose a single lift $\tilde a_i$ for each $i$. 
%$a$ is a union of images of $\tilde a_i, i=1, \dots, m$ under $\Gamma$. 
There are disjoint cusp neighborhoods $E_i$, $i=1, \dots, k$. 
For each $i$, choose a horodisk $\tilde E_i$ projecting onto $E_i$.
%Also, $E_i$ for $i=1, \dots, k$ form cusp neighborhoods in $\Sigma_+$, and 
Taking the antipodal images of $\tilde E_i$ and projecting, we 
obtain 
cusp neighborhoods $E_{i, -}$ in $\Sigma_-$ for $i=1, \dots, k$. 
We denote by $E$ the union of $E_i$ for $i=1, \dots, k$, and $E_-$ the union of $E_{i,-}$ for $i=1, \dots, k$. 

%Let $\partial \Ss_+$ denote the boundary of $\Ss_+$ with the boundary orientation.
%Then there is a domain $\tilde D_i= Z(\tilde a_i)$ in $\Ss_0$ for $i=1, \dots, m$ that is a union of 
%maximal geodesics in $\Ss_0$ tangent to $\partial \Ss_+$ in the boundary direction. 
%We also have 
%$\partial D_i = \tilde a_i \cup \tilde a_{i, -}$. 
%Then $\Gamma$ acts properly discontinuously on the union of images of $D_i$ under 
%$\Gamma$. On each $\tilde a_i$, there is a unique primitive element $\eta_i$
%acting on it in the orientation direction. It also acts on $D_i$. 

%We define $\tilde \Sigma := \Ss_+ \cup \Ss_- \cup \bigcup_{i=1}^m \Gamma({\tilde D_i})$
%where $\Gamma$ acts properly discontinuously, 
%and define $\Sigma := \tilde \Sigma/\Gamma$.  
%(See Section 5.1 of \cite{CDG22} and Theorem 5.3 of \cite{CG17}.)
%Then $\Sigma$ contains $\Sigma_+$ and $\Sigma_-$ and the complement is 
%a union $A$ of annuli $A_i$ for $i=1, \dots, m$, each of 
%which is an image of ${\tilde D_i}$.

%Here, $\Sigma$ is a union of 
%\[\Sigma_+, \Sigma_-, \hbox{ and } A_i, i=1, \dots, m.\] 
%\begin{itemize} 
%\item $\Sigma_+$,
%\item $\Sigma_-$, 
%\item $A_i$, $i=1, \dots, m$. 
%\end{itemize}

%We will now define ``compactification relative to parabolics"'. 

By the assumptions of Theorem \ref{thm:main} and Hypothesis \ref{hyp:positive}, 
$\Gamma$ acts properly discontinuously and freely on
$\Lspace \cup \tilde \Sigma$, and 
the quotient is a union of $M$ and a surface $\Sigma$. 
Let $p$ denote the covering map. (See Section 5.1 of \cite{CDG22}.) 

%\marginpar{\small A figure here}

%\marginpar{\tiny Itemize all these} 
By the definition of the relative compactification, 
a disjoint collection of solid tori 
$T_1, \dots, T_k$ are closed in $M_\romec:= (\Lspace/\Gamma) \cup \Sigma$, 
%for two complete hyperbolic surfaces $\Sigma_+, \Sigma_-$ with cusps  and 
and a compact $3$-manifold $M_\romecr:= M_\romec - \bigcup_{i=1}^k T_i^o$ is a deformation retract of $M_\romec$ such that
$M_\romecr \cap T_i$ is an annulus ${\mathfrak{A}}_i$ with boundary components 
$\alpha_i^+$ and $\alpha_i^-$ in $\Sigma_+$ and $\Sigma_-$ respectively. 
These have parabolic holonomies. 
Here, $\Sigma_e = \Ss_e/\Gamma$ for each $e=+, -$ denotes a complete hyperbolic surface 
where each cusp has the neighborhood bounded by 
$\alpha_j^{\rome}$ for some $j=1, \dots, k$.

Let $\eta_i$ denote the deck transformation corresponding to $\alpha_i^+$ acting on a preimage $\tilde \alpha_i^+$ of $\alpha_i^+$ for each $i$. 
In $\Lspace$, ${\mathfrak{A}}_i$ is covered by a disk $\tilde {\mathfrak{A}}_i$ whose boundary is the union of 
a horocycle in $\Ss_+$, its antipodal image in $\Ss_-$, and 
the segment $\zeta(p_i)$ associated to 
the parabolic fixed point $p_i$ of $h(\eta_i)$.
Now, $\tilde {\mathfrak{A}}_i$ is invariant under a one-parameter group $\Phi_{i,s}$, $s\in \bR$, of parabolic affine transformations containing $h(\eta_i)$. 
Also, $\tilde {\mathfrak{A}}_i$ bounds a $3$-cell $\tilde T_i$ covering $T_i$. 
The inverse image $p^{-1}({\mathfrak{A}}_i)$ is the union of 
disjoint copies of $\tilde {\mathfrak{A}}_i$ under $\Gamma$, 
and $\partial M_{\romecr}$ is covered by 
$(\tilde \Sigma  \setminus E) \cup 
\bigcup_{i=1}^k p^{-1}({\mathfrak{A}}_i)$.

\subsection{Deforming solid tori}\label{sub:deformtori} 
Now, we will try to make some of these solid tori 
homeomorphic to annuli times half-open intervals 
into compact solid tori by deforming their boundary
components. 
Now assume $h_\mu$ lies in a cone neighborhood $N$ of $h$ in $\mathcal{F}_n^{\rome}$
%for $\mu \in \mathcal{F}^{\romel}_n$.
for $\mu \in \Isom(\Lspace)^n$ 
described in Lemma \ref{lem:deformnop}.
%Let $N'$ be the complement in $N$ of 
%the set of representations with parabolic elements. 
Here, we assume $\mu$ is regarded as an element of $\bR^{6n}$ parameterizing $N$ where $\mu_0$ corresponds to the origin, and $N$ corresponds to a ball intersected with 
$\mathcal{F}_n^{\rome}$. 
(Note $\dim \Isom(\Lspace) = 6$.) 
Let $\mu_0$ correspond to $h$. 

We will choose a sufficiently small neighborhood in $N$ so that the completeness can be established. 
We now describe how these solid tori deform, 
and how we compactify them, ignoring $M_{\romecr}$ for a while: 

As we deform $\Gamma$, we change
$h(\eta_i)$ to a hyperbolic element $h_\mu(\eta_i)$ for 
each $i=1, \dots, k_v \leq k$ for $\mu\in N$.  
Note that we may not be deforming all of them to hyperbolic elements,
and we may assume the changing to hyperbolic elements 
happens for $i=1, \dots, k_v$ 
without loss of generality. 
We denote by $\Gamma_\mu$ the image of $h_\mu$ in $N$.

%There is a one-parameter family of hyperbolic elements 
%$\Phi_{i, s}(t)$ for $s > 0$ containing $h_t(\gamma_i)$ for each $t$ and $i$. 
%We have $\Phi_{i, s}(t) \ra \Phi_{i, s}$ for each $i=1, \dots, k, s\in \bR$ as $t \ra 0$. 

We first deform $T_i$ for $i=1, \dots, k$ bounded by ${\mathfrak{A}}_i$. %where
%each $T_i$ is diffeomorphic to ${\mathfrak{A}}_i\times [0, 1)$. 
%
Define 
\[\hat T_i :=(\Lspace \cup \Ss_+\cup \Ss_-)/\langle h(\eta_i)\rangle\] where $T_i$ embeds
as a submanifold  bounded by $\hat {\mathfrak{A}}_i$ that is the image of $\tilde {\mathfrak{A}}_i$. 

Here, we need to choose $N$ to be smaller if necessary for the argument to work. 

We define $\hat T_i(\mu):= (\Lspace \cup \Ss_+\cup \Ss_-)/\langle h_\mu(\eta_i)\rangle$. 
%We can choose as follows: 
We can choose a compact neighborhood $N_i$ of $\hat{\mathfrak{A}}_i$ in $\hat T_i$, 
and a compact neighborhood $N_i(\mu)$ of ${\mathfrak{A}}_i(\mu)$ in $\hat T_i(\mu)$ so that the following hold{\em :}
\begin{itemize}
\item There are connected 
preimages $\tilde N_i$ in $\Lspace \cup \Ss_+\cup \Ss_-$ of $N_i$
and $\tilde N_i(\mu)$ in $\Lspace \cup \Ss_+ \cup \Ss_-$ of $N_i(\mu)$.
\item There is a diffeomorphism $f_i(\mu): N_i \ra N_i(\mu)$ 
lifting to a diffeomorphism $\tilde f_i(\mu): \tilde N_i \ra \tilde N_i(\mu)$. 
\item $\tilde f_i(\mu) \ra \Idd$ in the $C^r$-topology for $r \geq 2$.
\end{itemize}
We construct $\tilde f_i(\mu)$ as a family of immersions 
as in  the proof of Thurston and Morgan 
written up by Lok \cite{Lok84}. 
We consider a cover of
$(\Lspace \cup \Ss_+\cup \Ss_-)$
by compact balls and define smooth isotopies 
starting from the compact sets in the intersections of the maximal number of balls. 
We can choose $\tilde f_i(\mu)$ so that it converges 
in the sense of $C^{r}, r \geq 2$ as 
$\mu \ra \mu_0$. 
Since $N_i$ is compact, it is covered by finitely many compact balls $B_{i, l}, l=1, \dots, l_i$ for some $l_i$. 
%where  $f_i(\mu)$ is an embedding:
%Since we are working with a regular neighborhood of a compact annulus, 
%we may use only two precompact balls intersecting at balls. 
We can use the argument of the proof of Theorem I.1.5.3 in \cite{CEG06} by finding $N_i(\mu)$ to show that 
%with $N(\mu_0)$ for $\mu$ sufficiently close to $\mu_0$ and considering the change as the change of the connections.
the induced map $f_i(\mu): N_i \ra N_i(\mu)$ is a diffeomorphism
for $\mu$ sufficiently close to  $\mu_0$.  (See also Remark I.1.5.5 in \cite{CEG06}.) Here, 
we may need to choose smaller $N$.
%since $f_i(\mu)$ injective in the balls $B_{i, k}$ for such $t$ and 
%if two pairs  $x_t, y_t$ of points of some distances are always going to the same points under $f_i(\mu)$, we will have the non-injectivity of $f_i$.
Hence, we will choose $N_i(\mu)$ as the image of $f_i(\mu)$ for 
$\mu$ sufficiently close to $\mu_0$.

Since $\tilde f_i(\mu)$ lifts $f_i(\mu)$, we have 
$h_\mu(\eta_i) \circ \tilde f_i(\mu) = \tilde f_i(\mu) \circ h(\eta_i)$. 
We define $\tilde {\mathfrak{A}}_i(\mu)$ as the image $\tilde f_i(\mu)(\tilde {\mathfrak{A}}_i)$. 
Now, $h_\mu(\eta_i)$ is hyperbolic for $i=1, \dots, k_v$. 
Since $\tilde {\mathfrak{A}}_i(\mu)$ covers a compact annulus ${\mathfrak{A}}_i(\mu)$ in $\hat T_i(\mu)$, 
$\tilde {\mathfrak{A}}_i(\mu)$ is properly embedded in $\Lspace \cup \Ss_+ \cup \Ss_-$
as follows from Lemma 3.4 of \cite{CG17}
using the action of hyperbolic cyclic group $\langle h_\mu(\gamma_i)\rangle$.
%$\tilde {\mathfrak{A}}_i$ is properly embedded in $\Lspace \cup \Ss_+ \cup \Ss_-$ 
Now, $\langle h_\mu(\eta_i) \rangle$ acts freely and properly discontinuously 
on $\tilde {\mathfrak{A}}_i(\mu)$. 

The set $\Ss_e \cap \tilde {\mathfrak{A}}_i(\mu)$, $e=+, -$, is an arc $\alpha_i^{\rome}(\mu)$ that bounds a convex disk 
$\tilde E_i^{\rome}(\mu) \subset \Ss_e$ for $i=1, \dots, k$. 
The boundary of $\tilde E_i^{\rome}(\mu)$ is a disjoint union of $\alpha_i^{\rome}(\mu)$ and 
a closed arc 
$\beta_i^{\rome}(\mu)$ in $\partial \Ss_e$ for $i=1, \dots, k_v$, and
is the closure of 
$\alpha_i^{\rome}(\mu)$ for $i = k_v +1, \dots, k$. 

%Let $\kappa_i$ for $i=1, \dots, p$, denote the deck transformation 
%acting on $\beta_i$. 
As we deform, $h_\mu(\kappa_i)$, $i=1,\dots, p$, 
acts on an open arc $\tilde \beta_i(\mu)\subset \partial \Ss_+$ where 
the quotient circle corresponds to the boundary component of $\Sigma_\mu$. 
%We denote by $\tilde \beta_i(\mu)$ the arc where $h_\mu(\kappa_i)$ acts on. 
We define $\tilde D_i(\mu) := Z(\beta_i(\mu))$ for $i= 1, \dots, p$. 
We construct the region $\tilde F_i(\mu):= Z(\beta_i^{+, o}(\mu))$ in $\Ss_0$.
We define 
\[\tilde \Sigma_\mu:= \Ss_+ \cup \Ss_- \cup 
\bigcup_{i=1}^p \Gamma_\mu(\tilde D_i(\mu))\cup 
\bigcup_{i=1}^{k_v} \Gamma_\mu(\tilde F_i(\mu)) \subset \SI^3.\]
$\Gamma_\mu$ acts properly discontinuously on it by Theorem 5.3 of \cite{CG17}. 
We define $\Sigma_\mu:= \tilde \Sigma_\mu/\Gamma_\mu$. 
%Here, we need $|\mu| < \eps_0$ for some constant $\eps_0 >0$ for this %argument to work. 

%Heuristically, $\tilde F_i(\mu)$ covers a compact annuli that is obtained by ``opening up'' the cusp.

%Recall from \cite{CG17} that there exists a real projective surface 
%$\Sigma_\mu$ in the boundary of $\Lspace$. 

We do not yet know that $\Gamma_\mu$ acts properly discontinuously and freely on $\Lspace$, 
and we need to construct some compact $3$-manifold. 

%so that 
%$\Lspace \cup \Sigma_\mu/\Gamma_\mu$ is a compact real projective manifold 
%with boundary $\Sigma_\mu/\Gamma_\mu$. 

By Theorem  5.3 of \cite{CG17}, 
the  closed surface $\Sigma_\mu$ is a union of 
\begin{itemize} 
\item two complete hyperbolic 
surfaces $\Sigma_{+, \mu}:= \Ss_+/\Gamma_\mu$, 
\item  $\Sigma_{-, \mu}:= \Ss_-/\Gamma_\mu$,  
\item annuli $A_{i, \mu}$, 
images of $\tilde D_i(\mu)$, $i=1, \dots, p$, and 
%covered by a union of oriented geodesic semicircles tangent to $\partial \Ss_+$ and $\partial \Ss_-$ where the initial vectors at $\partial \Ss_+$  
%are agreeing with the orientation of $\partial \Ss_+$. 
\item the annular images  of strips 
$\tilde F_i(\mu)$, which we denote by $F_{i, \mu}$,  $i=1, \dots, k_v$.  (``opened-up ones''.) 
\end{itemize} 
Here, $\partial \tilde D_i(\mu) = \beta_i(\mu) \cup \beta_{i}(\mu)_-$ 
covers the union of 
boundary components of 
the closure of $\Sigma_{+, \mu}$ and that of $\Sigma_{-, \mu}$. 
$\partial \tilde F_i(\mu)$ also covers the boundary components of these, 
which were opened up from cusps.  

%\marginpar{\small A figure here} 

Then the closure of $\tilde {\mathfrak{A}}_i(\mu)$ in $\Lspace \cup \Sigma_\mu$ bounds
a $3$-cell $E_i(\mu)$ so that $T_i(\mu):= E_i(\mu)/\langle h_\mu(\eta_i)\rangle$ is a compact solid torus. 
The boundary of $T_i(\mu)$ for each $i=1, \dots, k_v$ is a union of four annuli 
\begin{multline} 
{\mathfrak{A}}_i(\mu):= \tilde {\mathfrak{A}}_i(\mu)/\langle h_\mu(\eta_i)\rangle, 
E_i^{\rome}(\mu):= \tilde E_i^{\rome}(\mu)/\langle h_\mu(\eta_i)\rangle, \rome=\pm,
\hbox{ and } \\ 
F_i(\mu):= \tilde F_i(\mu)/\langle h_\mu(\eta_i)\rangle. 
\end{multline}
The boundary of $T_i(\mu)$ for each $i=k_v+1, \dots, k$, is a union of three annuli 
\[{\mathfrak{A}}_i(\mu):= \tilde {\mathfrak{A}}_i(\mu)/\langle h_\mu(\eta_i)\rangle 
\hbox{ and } 
E_i^{\rome}(\mu):= \tilde E_i^{\rome}(\mu)/\langle h_\mu(\eta_i)\rangle,
\rome = \pm.\]
% and $\clo(D_i(\mu))/h_\mu(\eta_i)$. 

%We call this $B_i(\mu)$. At the moment, we don't have the whole compactified manifold. 

\subsection{Deforming the compact part} \label{sub:deform}

%Now, we deform $M_{\romecr}$. We do this by deforming ${\mathfrak{A}}_i$ to 
%${\mathfrak{A}}_i(\mu)$. 
As we deform $\Gamma_\mu$, we will obtain the deformation of 
the real projective structure on $M_{\romecr}$ by first deforming the 
annulus ${\mathfrak{A}}_i$
to ${\mathfrak{A}}'_i(\mu)$ for each $i$. 

Again, we need to choose a smaller neighborhood in $N$ if necessary for the following argument to work. 

We apply $\tilde f_i(\mu)$ to $\tilde {\mathfrak{A}}_i$; 
more precisely, we apply it to the one-sided neighborhood 
$\tilde N_i$ of $\tilde {\mathfrak{A}}_i$ inside $\tilde M_\romec$ as follows: 

%This can be done following Lok's argument \cite{Lok} starting from compact sets meeting ${\mathfrak{A}}_i$.  
%We can do the deformation for some $\eps > 0$ depending on $\Gamma$
%so that 
%$d_{\Isom(\Lspace)}(g_i, g_i(\mu)) < \eps$ for $i=1, \dots, n$. 
%as the theory of Lok to work here.  

%Another way to see this is to use the section on the flat bundle:
We consider $\Isom(\Lspace) $ as a subgroup of $\SL_\pm(4, \bR)$ acting on $\SI^3$ projectively. 
For each representation $\rho$, 
we construct a flat $\SI^3$-bundle over $M_{\romec}$ by the product
$\tilde M_{\romec} \times \SI^3$ and taking the quotient by the diagonal action of $F_n$: 
\[\gamma(x, y) = (\gamma(x), \rho(\gamma)(y)) \hbox{ for } x \in \tilde M_{\romec}, y\in \SI^3, \gamma \in F_n,\]
where $\gamma(x)$ denotes the deck transformation on the universal cover 
$\tilde M_{\romec}$.
%For each representation $\rho$, 
%there is a flat connection by a foliation whose leaves are of form
%$\tilde M_{\romec} \times \{y\}, y\in \SI^3$, 
%transversal to each $\{x\} \times \SI^3$, $x\in \tilde M_{\romec}$.  
%and mapping 
%to $\tilde M_{\romec}$ as  a covering map under the projection. 
The geometric structure based on $(\SI^3, \SL_\pm(4, \bR))$ is given 
by a section transverse to the horizontal leaves of the flat connection. 
(See Goldman \cite{G88}).

At $\mu_0$, we have a transversal section. For $\mu$ in a 
sufficiently small neighborhood of $\mu_0$, we 
can still obtain a deformed $M_{\romecr}$ with 
the holonomy homomorphism $h_\mu$ deforming $\Sigma$ and ${\mathfrak{A}}_i$.  
We deform first for ${\mathfrak{A}}_i$ to become ${\mathfrak{A}}_i(\mu)$ 
while deforming  
$\Sigma - E^o$ to be mapped inside $\Sigma_\mu$, and we use 
a partition of unity to extend the transversal section everywhere on $M_{\romec}$. 
For deformations of holonomy representations $h_\mu$ sufficiently close 
to $h$, the transversality is preserved since 
sections change by a sufficiently small amount 
in the $C^r$-topology for $r \geq 2$. 
Alternatively, we can use 
the approach of Lok \cite{Lok84} starting from a locally finite collection 
of compact sets covering a precompact tubular neighborhood of ${\mathfrak{A}}_i$.
Here, we need to take a sufficiently small neighborhood of $\mu_0$ in $N$ if necessary. 
%Here, we need $|t| < \eps$ for some constant $\eps >0$ for this argument to work. 

%\marginpar{\tiny Is this sufficient?} 

We call the result $M_{\romecr}(\mu)$. Now, $M_{\romecr}(\mu)$ may not 
be a quotient of a domain 
in $\clo(\Lspace)$ unlike $M_{\romec}$.  This will follow later.

%\marginpar{\tiny some skipping here} 

Now, we glue $T_i(\mu)$ to $M_{\romecr}(\mu)$ by identifying ${\mathfrak{A}}_i(\mu)$ and ${\mathfrak{A}}'_i(\mu)$ in the natural way
for each $i=1,\dots, k$. 
The result is a manifold $M_{\romecrf}(\mu)$.  
If $k_v = k$, then 
it is compact 
since so are $M_{\romecr}(\mu)$ and $T_i(\mu)$ for $i= 1, \dots, k$.
Also, we can verify that $\partial M_{\romecrf}(\mu) = \Sigma_\mu$.  

\begin{proof}[Proof of Theorem \ref{thm:main}] 
By construction, $M_{\romecrf}(\mu)$, $\mu \in N$, 
is a flat Lorentzian manifold with a relative compactification. 
Hence, Theorem \ref{thm:Carriere} proves the result. 
\end{proof}

%\marginpar{\small A figure here} 

\subsection{Completeness using the argument of Carri\`ere} \label{sub:Carriere} 

A {\em geodesic} on a real projective manifold $N$ is a map defined on an interval such that 
the composition with a developing map of its lift to $\tilde N$ 
is a map into a great circle in $\SI^n$ where the image does not 
contain any pair of antipodal points. 
Recall that a real projective manifold $N$ is {\em convex} if every path is homotopic to a geodesic relative to endpoints.
(See \cite{psconv}.)

We begin the proof of Theorem \ref{thm:Carriere}: 
By an abuse of notation, 
let $\Sigma'$ be constructed as in Section \ref{sub:rcomp} using 
a free holonomy group $\Gamma' \subset \Isom(\Lspace)$ of rank $\geq 2$. 
Let $N$ be a flat Lorentzian manifold with the holonomy group $\Gamma'$. 
Let $N_{\romec}:= N \cup \Sigma'$ be a relatively compactified flat Lorentzian manifold, 
and let $N_{\romecr}:= N_{\romec} \setminus T$ for 
$T$ the union of solid tori with parabolic holonomy groups, 
where $N_{\romec}$ deformation retracts to $N_{\romecr}$.  
Since $N_{\romecr}$ is compact, its universal cover 
$\tilde N_{\romecr}$ has a compact fundamental domain 
$F_{\romecr}$. 

We will show that $N_{\romec}^o$ is covered by $\Lspace$. 
To achieve that, we will prove that $N_{\romec}^o$ is convex. 
%(See \cite{psconv}.) 

%An {\em affine manifold} is a manifold with an affine structure, i.e., 
%a geometric structure modeled on an affine space with affine transformation 
%group acting on it. Hence, $M$ and $M_{\romecrf}$ are affine manifolds. 

As its developing map 
$\dev: \tilde N_{\romec} \ra \clo(\Lspace)$ is 
a smooth immersion,
the restriction
$\dev|\tilde N_{\romec}^o$ mapping to $\Lspace$ is also 
a smooth immersion. 
Let $h:\pi_1(N) \ra \Isom(\Lspace)$ denote the holonomy homomorphism. 
For now, assume that the holonomy group $\mathcal{L}\circ h(\pi_1(M))$ is in $\SO_+(2, 1)$.

In this paper, a ``metric'' means a distance metric. 
Again, using the theory of \cite{psconv}, we obtain a Kuiper completion 
$(\tilde N_{\romec} ^o)\che{}$ where we complete the path-metric
induced from the pulled-back Riemannian metric of the standard Euclidean Riemannian 
metric of $\Lspace$. 
The induced path-metric will be called the {\em Euclidean metric} on $\tilde N_{\romec}^o$. 
We denote by $\delta_\infty \tilde N_{\romec}^o$ the set of ideal points
of $\tilde N_{\romec}^o$.
The developing map extends to $(\tilde N_{\romec}^o)\che{}$; 
we continue to denote the extension by 
$\dev$. 
In particular, $\dev(\delta_\infty \tilde N_{\romec}^o)$ is 
contained in $\Lspace$. 

%We will show that it is a homeomorphism to a convex open domain in $\Lspace$. 
From here on, we fix $\mu$. 
Suppose that $\tilde N_{\romec} ^o$ is not convex. 
Then there exists a triangle $T$ in $(\tilde N_{\romec} ^o)\che{}$ 
with all three vertices and two edges in $\tilde N_{\romec} ^o$
and the interior of
 one edge $E_v$ meeting $\delta_\infty \tilde N_{\romec} ^o$.
Let $v$ be the opposite vertex. 
The two other edges are disjoint from $\delta_\infty \tilde N_{\romec} ^o$. 
(See \cite{CG93}, where 
we have a slightly different metric but the arguments 
are the same.)
%Then  $E$ meets an ideal point.

% \marginpar{\small notation $E$ problem... abuse.}

%Let $v$ denote the vertex opposite $E'$. 
Consider the maximal domain $D_v$ in $\tilde N_{\romec} ^o$
consisting of points that can be reached by a segment from $v$. 
Since $\dev$ is an immersion, $D_v$ is an open domain.
Clearly, $\dev| D_v$ is an embedding into $\Lspace$. 
$T$ is in the closure $\clo(D_v)$. 
By a perturbation of $T$ inside $\clo(D_v)$ as in Proposition 4.3 of \cite{psconv}, 
we can assume that there exists a triangle $T'$ that meets 
$\delta_\infty \tilde N_{\romec} ^o$. 
We may assume that the intersection occurs only at 
a single interior point $p'$ of an edge say $E'_v$
by perturbing the edge $E_v$ in $\clo(D_v)$. 
Now, we let $T$ be this $T'$ and $E_v$ be this $E'_v$ by relabeling. 

Let $p:\tilde N_{\romec}  \ra N_{\romec} $ denote the universal covering map. 
Let $l$ be a maximal segment in $E_v$ ending at an ideal point. 
Then $p: l\ra N_{\romec} $ maps to an arc of infinite length 
with respect to any complete Riemannian metric of $N_{\romec} $. 
We claim that $l$ cannot be in the preimage of $T_i$ for any $i=1, \dots, k_v$ 
after some point. Since each of the components of the preimage is 
a closed subset of $\Lspace \cup \Ss$, 
$\dev|l$ must end at a point of $\Ss$. 
This contradicts the fact $\dev(p') \in \Lspace$. 
Hence, $p| l$ must enter and leave $N_{\romecr} $ infinitely many times. 
Hence, there is a sequence $\{p_i\in l\}$ and an 
unbounded sequence $\{g_i \in F_n\}$ such that 
$g_i(p_i) \in F_{\romecr}$.

We can also form a path-metric on $\tilde N_{\romec} $ from the induced Riemannian 
metric from the standard spherical Riemannian metric of $\SI^3$ 
using the developing map. 
The induced metric is called a {\em spherical metric}. 

\begin{lemma}\label{lem:compact} 
Let $N'$ be a real projective manifold  with totally geodesic boundary. 
Suppose that 
the interior of its universal cover $\tilde N'$ develops into $\Lspace$. 
Let $K$ be a compact subset of $\tilde N'$. 
Then there exists $\eps > 0$ such that each
ball of Euclidean radius $\eps$ with 
center $x \in K \cap \tilde N^{\prime o}$ is mapped by 
the developing map 
to a Euclidean ball of the same radius in $\Lspace$.  
\end{lemma} 
\begin{proof} 
We can cover $K$ by open sets such that on each
the developing map is a homeomorphism.  
Choose $\delta >0$ that is a Lebesgue number of 
this covering for the spherical metric. 

We can verify easily by a direct computation that 
the spherical metric on $\Lspace$ is bounded above by 
twice the Euclidean metric, 
and the same inequality holds on $\tilde N^{\prime o}$. 
Let $\eps > 0$ be less than $\delta/2$. 
Then any directed ray from $x$ in any direction can be extended up to a 
length $<\eps$ since it is contained in one of the open balls of the cover.
%with respect to $\tilde N^{\prime o}$. 
The conclusion follows. 
\end{proof}

%Let $P$ be the union of 
%components of the preimage
%of $T_i, i = k_v+1, \dots, k$. 
%% that intersect $F_\romecr$. 
%%Let $P'$ be the union of components 
%%of the preimage of $T_i, i = k_v+1, \dots, k$ disjoint from $F_\romecr$. 
%Let $F_{\romecr, 2}$ denote the union of
%$F_{\romecr}$ and the
% images of $F_{\romecr}$ 
%meeting $F_{\romecr}$ in $\tilde N_{\romecr} $. 
%%Let $F'$ denote the topological interior of $F_{\romecr, 2}$ 
%%in $\tilde N_{\romecr} $. 
%There is a lower bound $\eps' >0$ for the distance from
%$F_{\romecr}$ 
%to $(\tilde N_{\romec}  \setminus F_{\romecr, 2}) \setminus P$ 
%with respect to the spherical metric. 
%%Hence, there is a lower bound $\eps> 0$ for the induced Euclidean metric $d_E$ on 
%%$\tilde N_{\romec} $: 
%There is a lower bound on the distance 
%from $F_{\romecr}$ to $\tilde M_{\romecr} \setminus F_{\romecr, 2}$ 
%with respect to the induced spherical metric on 
%$\tilde N_{\romec} $ 
%since one is compact
%and the other has a disjoint closure  on $\tilde N_{\romec} $.
%%with induced spherical metric.  
%On $\Lspace$, twice the Euclidean metric dominates the spherical metric. 
%Hence, this will hold on $\tilde N_{\romec} $, and there is 
%a lower bound on the Euclidean distance as well. 

Choose $\eps>0$ smaller than half of the constant from Lemma 
\ref{lem:compact} for the compact set $F_{\romecr}$,
which is the compact fundamental domain form $N_{\romecr}$. 
Now, we consider the union $F''$ 
%of $F_{\romecr, 2}$ and 
of all the compact balls of radius 
$\eps/2$ centered at points of 
$F_\romecr \cap \tilde N_{\romec}^o$ 
%$F_\romecr \cap P \cap \tilde M_{\romecrf}^o$ 
under the Euclidean metric. 
By Lemma \ref{lem:compact}, these balls all develop 
to Euclidean balls of same radius.

%where $P$ is the union of 
%components of the inverse images 
%of $T_i, i = k_v+1, \dots, k$ that intersect $F_\romecr$. 
We claim that $F''$ is compact in $\tilde N_{\romec} $: 
This follows from the fact that any sequence of closed Euclidean balls $B_i$ of radius $\eps/2$ centered at $q_i \in F_{\romecr}$ geometrically converges to a singleton where $q_i$ converges to. 
Hence, any sequence of points of $F''$ has a convergent subsequence. 

Now $g_i(p_i)$ is contained in a compact Euclidean ball  $E_i$ in $\tilde N_{\romec} ^o$ centered at $g_i(p_i)$ of radius $\eps/2$.
%since $\eps$ is chosen by Lemma \ref{lem:compact}. 
%This ball is contained in $F''$: 
%If $x \in E_i$ is in $\tilde N_{\romecr} $, then $x$ must be in $F_{\romecr, 2}$ by
%the distance conditions. If $x$ is in $P$, then there is a 
%directed ray from 
%$g_i(p_i) \in F_{\romecr}$ to $x$. 
%This path cannot meet $\tilde N_{\romecr}  \setminus F'' \setminus P^o$ %since we have a lower bound $\eps$. 
%If it meets $P$, then $x$
%must be inside one of the added Euclidean balls in the previous paragraph. 
%Hence, $E_i \subset F''$. 
Also, $E_i$ develops to a Euclidean ball of radius $\eps/2$
by our choice in Lemma \ref{lem:compact}. 
%this follows. 
%We take a directed ray from $x$ towards any direction. It can only pass 
%$F_{\romecr, 2}$ or $P$. By our choice in Lemma \ref{lem:compact}, 
%this follows. 

An {\em ellipsoid} on a flat Lorentzian manifold is a compact $3$-ball whose lift
in the universal cover develops to an ellipsoid in $\Lspace$. 
The {\em center} of an ellipsoid is the point corresponding 
to the center of the developed image. 

Denote by $d_{\mathrm{E}}$ the Euclidean metric on $\tilde N_{\romec}^o$ 
 induced from the Euclidean metric of $\Lspace$.  
As Carri\`ere \cite{Carr89} showed, 
$g_i^{-1}(E_i)$ is a sequence of ellipsoids with centers $p_i$ 
so that the Euclidean diameter $\diam_{d_{\mathrm{E}}}(g_i^{-1}(E_i)) > \delta$ for $\delta> 0$. 
This follows since the discompactedness of the linear 
part in $\SO(2,1)$ in $\Isom(\Lspace)$ is $1$. 
Hence, $g_i^{-1}(E_i)\cap T$ 
contains a segment $s_i$ with an endpoint $p_i$ of Euclidean length at least 
$\delta> 0$. 

Hence, $g_i^{-1}(E_i) \cap T$ contains a point $x_i$ of $d_{\mathrm{E}}(x_i, p_i) =\delta/2$ 
converging to $x \in T \setminus \delta_\infty \tilde N_{\romec} ^o$ after passing to a subsequence 
 by the geometry of the triangle.  
 Since $g_i(p_i)$ is in the compact ball of 
radius $\eps/2$ in a compact $F''$, 
this is a contradiction to the proper discontinuity of the action of the fundamental group
on $\tilde M_{\romecrf}$. 
%Since $g_i^{-1}(F_2)$ contains $g_i^{-1}(E_i)$, there is a sequence of points $x_i \in F_2$ 

%We remark that we still have compact fundamental domain and the argument still works
%although the fundamental domain is not in $\Lspace$. 

%\marginpar{\small A figure here} 

%\marginpar{$\Omega$ and $\Omega_t$ define these?} 
Therefore, $N_{\romec} ^o$ is convex and $\dev|\tilde N_{\romec} ^o$ is an embedding into 
a convex domain $\Omega$ in $\Lspace$. %This domain is $\Gamma$-invariant. 

The completeness as an affine manifold follows: 
%$\Omega$ contains the domains $\tilde E_i$. 
%Hence $\clo(\Omega)$ contains the strips $\tilde {\mathfrak{A}}_i$ and %its $\Gamma$-images. 
%The set of $\Gamma$-images of these points is contained in 
%$\clo(\Omega)$. 
Since $\tilde F_i$ and $\tilde D_i$ are in 
$\tilde N_{\romec} $, 
%we see that the fixed points of 
%$h(g_i)$ and $h(\gamma_i)$ in $\partial \Ss_+$ and its antipodes in 
%$\partial \Ss_-$ are in the closure of 
their images are in the closure of 
$\Omega= \dev(\tilde N_{\romec} )^o$. 
%Also, the fixed points of its conjugates are in it also. 
%The closure of the set of these points is the limit set $\Lambda$.
%From the properties of the Fuchian groups, the convex hull of $\Lambda$ 
%is a convex domain $C$, and 
%so is the convex hull of the limit points in $\partial \Ss_-$ which must equal 
%$\mathcal{A}(C)$. The convex hull in $\clo(\Lspace)$ of
%$C \cup \mathcal{A}(C)\cup \Omega$ is $\clo(\Lspace)$ as we can take the interior point $x$ 
%of $C$ and its antipodes and their neighborhoods in $C$ and $\mathcal{A}%(C)$. 
Since $\clo(\Omega)$ is a convex domain in $\clo(\Lspace)$,
containing these subsets of $\Ss_0$, 
%containing $C$ and $\mathcal{A}(C)$, 
$\clo(\Omega)$ contains all the limit points of $\Ss_+$ and $\Ss_-$ of the Fuchsian group in $\partial \Ss_+$ and $\partial \Ss_-$ as well as 
$\zeta(x)$ for every limit point $x$ in them.
There are at least two disjoint geodesic arcs $\zeta(x)$ and $\zeta(y)$ for limit points $x$ and $y$ in $\Ss$ that cannot lie in 
any hemisphere. 
Since a convex hull a set in $\Ss$ is the complement of the union of all hemisphere disjoint from it, 
their convex hull must be the entire
$\Ss$. 
The only possibility is that $\clo(\Omega) = \clo(\Lspace)$.
Hence, $\Omega =\Lspace$. 

%\marginpar{\small we need that limit set and the anitpodal image is in the closure here.} 

%\marginpar{\tiny Jan 10th 12:30am.. I am a bit lost here... } 

%Final version for $\SO(2, 1)$ not just for $\SO(2, 1)_+$ here... 
%TO BE DONE... 
\subsection{General case of $\Isom(\Lspace)$} \label{sub:generalcase} 

%We will now consider a Margulis space-time $M$ 
%with parabolic elements where 
Now, suppose that 
$\mathcal{L}\circ h(\pi_1(N)$ is not contained in $\SO(2, 1)_+$. 
Assume that $\pi_1(N)$ is identified with 
$F_n$ for a free group $F_n$ of rank $n$.
We can write $F_n=\langle F',\phi\rangle$, where $F'\le F_n$ is an index-two subgroup, 
$\phi \notin F'$, and $h(F')\subset \Isom_+(\Lspace)$
as in the proof of Theorem \ref{thm:nonpara}.
%Note there are deformations $h$ giving us 
%$h|G$ and $h(g)$ for $\mu\in N$ where 
%$N$ is a sufficiently small neighborhood of $\mu_0$ in 
%$\mathcal{F}^{\romel}_n\cup \{\mu_0\}$.
%(From now on, a `neighborhood' means 
%a neighborhood in $\mathcal{F}^{\romel}_n\cup \{\mu_0\}$.)
Now, $\tilde N/F'$ covers $\tilde N/F_n$ by a two-to-one map where
the deck transformation group is generated by a map 
$\hat \phi: \tilde N/F' \ra \tilde N/F'$
induced by $\phi$, and is
isomorphic to $\bZ_2$. 
%We denote by $G$ the group deformed from $G$.
%We denote by $M$ the space $\Lspace/\Gamma$ and define
%$M:= \Lspace/h(F_n)$ for $\mu \in N$ where $N$ is a sufficiently 
%small neighborhood of $\mu_0$. 

%%Since $G$ is a free group of rank $\geq 2$, we can apply the results of Sections 
%\ref{sub:deformtori} and \ref{sub:deform}
%to $M_{\mathrm d}=\Lspace/h(G)$ to obtain a deformed complete Margulis space-time 
%$M_{\mathrm d}:=\Lspace/h(G)$ without parabolics for $\mu \in N$. 

%Since $M_1$ is double covered by $M$, 
Now, $\hat \phi$ is an involution 
on $N_{\mathrm d}:= \tilde N/F'$ without fixed points. 
%Let $\hat g$ denote the generator of this group. 
%We can choose the parabolic tori $T_1, \dots, T_k$
%so that each one is sent to a different one since $\hat \phi$ 
%does not preserve a parabolic end of $\Sigma_+$
%by Section 5.1 of \cite{CDG22}.  
%We denote by $T$ the union of $T_i$, $i=1, \dots, k$.
%We may assume that $T_1, \dots, T_{k/2}$ are not exchanged by $\hat \phi$
%and are sent to $T_{k/2 +1}, \dots, T_k$ where $k$ is even. 
%Let $T$ denote the image in $N$ of the union of 
%$T_i$, $i=1, \dots, k$.
%We identify the boundary of the 
%compact manifold $(N_{\mathrm d}\setminus T^o)/\langle \hat \phi \rangle$
%with $\partial T_1 \cup \cdots \cup \partial T_{k/2}$. 
$N_{\mathrm d}$ admits a relative compactification with respect to solid tori
$T_1, \dots, T_k$ which are preimages of ones in $N$. 
Hence, Section \ref{sub:Carriere} shows that 
$N_{\mathrm d}$ is complete, and so is $N$. 
This completes the proof of Theorem \ref{thm:Carriere}. \qed

\bibliographystyle{plain} 
\bibliography{cdg}

\end{document}

%% file: lstylesCDG.tex
\newcommand\bR{{\mathbb{R}}}

\newcommand\bZ{{\mathbb Z}}

\newcommand\Hom{{\rm Hom}}
\newcommand\dev{{\bf dev}}
\newcommand\SI{{\mathbb{S}}}

\newcommand\clo{{\rm Cl}}

\newcommand\ra{\rightarrow}

\newcommand\che{\check}

\newcommand\eps{\epsilon}

\newcommand\Aff{{\mathbf{Aff}}}

\newcommand\Idd{{\rm I}}

\newcommand\Isom{{\mathbf{Isom}}}

\newcommand\SL{{\mathsf{SL}}}

\newcommand\SO{{\mathsf{SO}}}

\newcommand\Ss{{\mathbb{S}}}

\newcommand\diam{{\mathrm{diam}  } }

\newcommand{\deltxt}[1]{}
\newcommand{\KFAR}[1]{}

\newcommand\Lspace{\mathsf E}

\newcommand\PO{{\mathsf{PSO}}}

\newcommand\Bs{{\mathsf{B}}}

\newcommand{\va}{\mathbf{a}}
\newcommand{\vb}{\mathbf{b}}
\newcommand{\vc}{\mathbf{c}}

\newcommand{\bg}{\mathsf{g}}